\documentclass[11pt]{amsart}

\usepackage{amsthm}

\usepackage{amsmath,latexsym}
\usepackage{amssymb}

\usepackage{graphicx}

\usepackage[vlined,ruled]{algorithm2e}
\usepackage{array,multirow,makecell}

\usepackage{booktabs}
\usepackage{hyperref}
\def\R{\mathbb{R}}
\def\Z{\mathbb{Z}}
\def\Eff{\mathrm{Eff}}

\usepackage{caption}
\newtheorem{theorem}{Theorem}[section]
\newtheorem{lemma}[theorem]{Lemma}
\newtheorem{definition}[theorem]{Definition}
\newtheorem{proposition}[theorem]{Proposition}
\newtheorem{example}[theorem]{Example}
\let\origmaketitle\maketitle
\def\maketitle{
  \begingroup
  \def\uppercasenonmath##1{} 
  \let\MakeUppercase\relax 
  \origmaketitle
  \endgroup
}

\begin{document}
\title[An improved algorithm to compute the $\omega$-primality]{\large An improved algorithm to compute the $\omega$-primality}

\author[Wissem Achour, Djamal Chaabane \MakeLowercase{and} V\'ictor Blanco]{{\large Wissem Achour$^\dagger$, Djamal Chaabane$^\dagger$ and  V\'ictor Blanco$^\ddagger$}\medskip\\
$^\dagger$Dept. of Operations Research,  USTHB\\
$^\ddagger$IEMath-GR, Universidad de Granada\\
}

\address{Dept. of Operations Research, University of Science and Technology Houari Boumediene, ALGERIA.}
\email{wissem\_achour16@yahoo.com; dchaabane@usthb.dz}

\address{IEMath-GR, Universidad de Granada, SPAIN.}
\email{vblanco@ugr.es}

\date{\today}


\begin{abstract}
In this paper, we present an improved methodology to compute $\omega$-invariant of numerical semigroup. The approach is based on adapting a recent  resolution method for optimizing a linear  function  over the set of efficient solutions of a multiple objective integer linear programming problem.  The numerical experiments show the efficiency of the proposed technique compared to the existing methods.
\end{abstract}


\maketitle

\section{Introduction}
\label{intro}
The problem of optimizing a function over the set of efficient solutions of a multiobjective optimization problem has attracted the attention of researchers in the last decades due to the wide range of applications of this kind of problems. Needless to say that most (if not all) of the real-world problems, in different fields (manufacturing, logistic, finance, ...), have sometimes to deal with several conflicting goals. Several techniques have been designed to solve the different types of problems under this framework (see e.g., \cite{ehrgott,gandibleux,steuer}). However, in general, the solution set of these problems is large, which is not reasonable from the decision maker viewpoint. Hence, optimizing a function over the set of solutions of a multiobjective function becomes a very useful tool to select the ``best'' solution among that set (see \cite{shigeno} for a recent real-world application of this problem). But if multiobjective optimization is difficult itself, optimizing over the unknown set of efficient solution is a challenge in the mathematical programming community.

The continuous linear case have been widely studied (see \cite{philip,benson,ecker1,sayin,yamamoto}, among others). However, there is a lack of literature analyzing the optimization of a linear function over the efficient set of a multiobjective discrete problem, and relatively few algorithms have been proposed to solve them (see\cite{abbas,mebrek,jorge,boland}).   

As mentioned above, there are many applications of optimizing a function over the set of efficient solutions of a multiobjective problem. In this paper we analyze one of them arising in a more abstract field: the Algebraic Theory of Numbers. In particular, we study an algebraic invariant that appears when trying to identify the structure of a Numerical Semigroup. A Numerical Semigroup is a subset of the set of nonnegative integer numbers, which represent all the nonnegative integer combinations of certain set of integers. The study of these objects allows to understand the insights of some interesting Diophantine Equations, which are closely related with knapsack polytopes. We study here the $\omega$-invariant of a Numerical Semigroup, an invariant proposed in \cite{geroldinger}, and that has paid a lot of attention in the last years from researchers form Computational Algebra (see\cite{halter}\cite{blanco}\cite{omidali}, among others). 

In \cite{bvictor}, the author formulates the problem of computing the $\omega$-invariant as a problem of optimizing a linear function over the set of efficient solutions of a discrete multiobjective problem. Hence, a methodology based on \cite{jorge} was applied to solve the problems. The main advantage of this approach with respect to those proposed before is that since it uses mathematical programming tools it allows one to avoid the complete enumeration of the discrete feasible set.

In this paper, we propose a different strategy for solving the problem in which we exploit the structure of the problem by computing new bounds for the problem that allows to strengthen the search of the solution. This simple idea allows us to speed up the method and solve the problem in smaller times that those consumed by the one in \cite{bvictor}.

The paper is organized in five sections. In Section \ref{notation}, we briefly recall some preliminary notions on Multiobjective Optimization and Numerical Semigroups, needed for the development in the rest of the paper. Section \ref{description} is devoted to describe our approach for computing the $\omega$-invariant.  In Section \ref{num-exp} we report the results of some computational experiments. Finally, in Section \ref{conclusion} we give some concluding remarks and perspectives.

\section {Preliminaries}\label{notation}

\subsection{Optimization over the Efficient set}
We are given a finite set of linear objective functions induced by the $p$ vectors $c_1, \ldots, c_p \in \R^n$, and a polyhedral feasible set, $P$, defined by means of linear equations $P=\{x\in \R_+^n: Ax = b\}$, where $A\in\R^{m\times n}$ and $b \in \R^m$. A multiobjective linear integer programming problem  is written as:
\begin{align}
\max & \;\; C x := (c_1(x), \ldots, c_p(x))\label{molip}\tag{${\rm MOLIP}$}\\
\mbox{s.t. }& x\in P \cap \Z_+^n. \nonumber
\end{align}

Solving \eqref{molip} means finding the set of the so-called \textit{efficient} solutions. A feasible solution $x^*\in P \cap \Z_+^n$ is said efficient solution of $\eqref{molip}$ if and only if there is no feasible solution $x \in D $ such that $C x \geq C x^{*}$ with  $C x \neq C x^{*}$. Otherwise,  $x^{*}$ is not efficient and $C x^{*}$ is said to be dominated by $C x$. The set of efficient solutions for the problem above will be denoted by $\Eff(C,P)$.
Despite finding the whole set of efficient solutions is hard since the problem is \#P-hard~\cite{EG02}, one can check whether a given feasible solution $\bar x \in P \cap \Z^n_+$ is an efficient solution or not solving an  integer linear programming problem. Ecker and Kouada \cite{ecker} stated that one can solve the following problem:
\begin{align}
\Psi(\bar x) = \max & \;\;\sum\limits_{i=1}^{p}\psi_{i}\nonumber\\
\mbox{s.t. }& c_j x+\psi_j= c_j \tilde{x}, \forall j=1, \ldots, p,\label{EK}\tag{$\rm{EK(}\tilde{x}\rm{)}$}\\
&x \in P \cap \Z_+^n, \nonumber\\
&\psi_{j} \geqslant  0, \forall j=1, \ldots, p.\nonumber
\end{align}
Then, $\bar {x}$ is an efficient solution if and only if $\Psi(\bar x) = 0$. Furthermore, in case $\bar{x}$ is not efficient, the optimal $x$-variables of \eqref{EK} provide an efficient solution dominating $\bar{x}$.
In this paper, we analyze beyond multiobjective optimization, since we analyze the problem of optimizing a linear function over the set of efficient solutions of a multiobjective linear integer programming problem. The problem can be formulated as follows:
\begin{align}
\max  & \; \phi(x) := d x\nonumber\\
\mbox{s.t. }&x\in \Eff(C,P). \label{Main}\tag{${\rm OES}$}
\end{align}
where $d =(d_1, \ldots, d_n) \in \R^n$.

A possible relaxation of the problem above consists of relaxing the requirements of efficiency for $x$:
\begin{align}
\max ,  & \;\phi(x) \\
\mbox{s.t. }& x\in P \cap \Z^n_+. \label{OESR}\tag{${\rm OES}_R$}
\end{align}

\subsection{The $\omega$-invariant}

In what follows we describe the problem analyzed in this paper. It is defined over the algebraic structure of numerical semigroups. Hence, we first describe the basic elements under this theory and the notions needed to state the problem of computing the $\omega$-primality of a numerical semigroup. 

$S\subseteq \Z_+$ is said a numerical semigroup if: (1) $0 \in S$; (2) $\Z_+ \backslash S$ is finite; and (3) $S$ is closed for the sum ($x + y \in S$ for all $x, y \in S$). A set $\{n_{1},\ldots, n_{p}\} \subset S$ with $\gcd = (n_{1}, \ldots, n_{p})= 1$ is system of generators of $S$ if $S =\left\{s \in \Z_+: s = x_{1}n_{1} + ...+ x_{p} n_{p}, x_1, \ldots, x_p \in \Z_+ \right\}$. The set $\{n_{1}, \ldots, n_{p}\}$ is a minimal system of generators if no proper subset of it generates $S$. It is not difficult to see that every numerical semigroup has a unique minimal generating system \cite{rosales}. If $\{n_1, \ldots, n_p\}$ is a minimal system of generators of $S$, we denote $S=\langle n_1, \ldots, n_p\rangle$.

One interesting question when analyzing numerical semigroups is the study of decompositions of the elements of the semigroup with respect to the minimal generating system, that is, its set of factorizations. Let $S=\langle n_1, \ldots, n_p\rangle$ be a numerical semigroup, and $s\in S$. The possible factorizations of $s$ are the different ways of obtaining $s$ as an integer linear combination of the minimal system of generators, i.e., the solutions of the following system of diophantine equations:
\begin{align*}
s= \sum_{i=1}^p n_i x_i,
\end{align*}
with $x_1, \ldots, x_p \in \Z_+$.

The solutions of the above system is denoted by $F(s)$, the set of factorizations of $s$ in $S$.

In the analysis of factorization, Gerolding and Hassler \cite{gerol,geroldin} introduced the notion of $\omega$-primality which allows to measure how far is an element of the semigroup of  being uniquely factorized. A lot of attention has been paid in the recent years to the computation and analysis of such an algebraic invariant (see~\cite{anderson,ander,gero,you}, among others). Observe this measures is closely related with some integer programming models, since unique factorizations will provide single feasible solutions to knapsack problems.

\begin{definition}[Omega Invariant]
 Let $S = \langle n_{1},\cdots, n_{p}\rangle$ be a numerical semigroup, and $s \in S$. The $\omega$-invariant of $s$ in $S$, $\omega(S, s)$, is the smallest nonnegative integer $K$, such that
for each finite set of elements $\{s_1, \ldots, s_n\} \subseteq S$ that verifies $\sum_{i \in S_K} s_i - s \in S$, there exists $\Omega \subseteq \{1,\ldots,n\}$ with cardinality $K$ and such that $\sum_{i \in \Omega} s_i - s \in S$. The global $\omega$-invariant of the semigrouo $S$ is defined as $\omega(S)= \:\max ,  \{ \omega(S, n_{i})\: : \:  i= 1,\cdots, p \}$.
\end{definition}
With such an invariant, $s$ is prime, i.e. uniquely factorized if $\omega(S,s)=1$. Otherwise, if $\omega(S,s)>1$, then $\omega(S)>1$ and the semigroup is said $\omega(S)$-prime. In \cite{blanco}, is is proved that:
$$
\omega(S, s) = \max \left\{ \displaystyle \sum_{i=1}^{p} x_{i} \: : \: x \in  {\rm Minimals} \{F(s+s'): s' \in S\} \right\}.
$$
(here, ${\rm Minimals}\{Q\}$ stands for the set of component-wise minimal elements of $Q \subset\R^p$.)

This equation states that one can compute the $\omega$-invariant by searching among the set of minimal elements of certain set. 

The interesting point here is the connection between the $\omega$-invariant and mathematical programming. In particular, from an algebraic result in \cite{blanco}, in  \cite{bvictor} the author formulates the problem as a problem with the shape of \eqref{Main}, i.e. of optimizing a linear function over the efficient set of a multiobjective linear integer programming problem:
\begin{lemma}[\cite{bvictor}]
Let $S=\langle n_1, \ldots, n_p\rangle$ and $s \in S$. Then, 
  \begin{align} \omega(S,n_j)= \max & \sum_{i=1}^n x_i\\
\mbox{s.t. }& (x, y) \in \Eff(C,P)
\end{align}
where $C=(Id_{p},0_p)$, and $P=\{x\in \R_+^n: Ax = b\}$ with $A=(n_1,\ldots,n_p,-n_1,\ldots,-n_p)$ and $b=n_j$ (here, $Id_p$ stands for the $p\times p$ identity matrix and $0_p$ is the $p\times p$ matrix with all its entries zero.
\end{lemma}
The aim of this paper is to efficiently compute the $\omega$-invariant of a numerical semigroup.

\section{A solution scheme for the problem}\label{description}

\noindent In this section we propose a new mathematical programming approach for computing the $\omega$-invariant of a numerical semigroup. We show that our new algorithm provides the  $\omega(S)$ without an explicit enumeration of all efficient factorizations in a finite number of iterations. In \cite{bvictor} the author proposed the first mathematical optimization model and resolution approach for the problem. However, a major drawback of the  resolution presented in \cite{bvictor} is that it does not fully exploit the specific nature of the problem. Here, we propose an improved algorithm that incorporates information about the algebraic structure of the problem.

Let  $S= \langle n_{1},...,n_{p} \rangle$ be a numerical semigroup. The initial step consists of generating a subset of $p-1$ efficient solutions for the multiobjective problem, which is based on the following result whose proof is straightforward.

\begin{lemma}
Let $j\in\{1, \ldots,p\}$. Then, for $k \in \{1,\ldots,p\}\backslash\{j\}$, the solution of the following integer programming problem: 
\begin{align}
\overline x_k^j = \min & \;\; x_{k}\nonumber\\
\mbox{s.t. } & n_{k} x_{k} -\displaystyle \sum_{i=1}^{p} n_{i}y_{i}= n_{j}\label{p0}\tag{${\rm UBound}_{jk}$}\\
&x_k \in \Z_+,  y \in \Z_+^{p}\nonumber
\end{align}
is in $\Eff(C,P)$.
\end{lemma}
The above result allows us to provide an initial lower bound for $\omega(S,n_j)$: $UB_j := \max_{k\neq j},   \left\{ \sum_{i=1}^{p}\overline{x}_{i}^{k}  \right\}$, with associated efficient solution that we denote by $\underline{x}^j$. Consequently, we have also an upper bound for $\omega(S)$, $UB=\displaystyle\max_{j=1,\ldots,p} UB_j$.

Next, in order to find another solution not dominated by $ \underline{x}^{j}$ we use the Sylva and Crema's cut  \cite{sylva2004,sylva2007}, which consists on solving the following integer linear programming problem:
\begin{align}
\max& \sum_{i=1}^{p} x_{i} \label{psj}\tag{${\rm PS}(\overline{x}_{j})$}\\
\mbox{s.t. } & x_{i} \leq k_{i}(\underline{x}_{i}^{j}-1)- M_{i}(k_{i}-1)&i=1, \ldots, p,\label{c1}\\
&\displaystyle \sum_{i=1}^{p}n_{i}x_{i} -\displaystyle \sum_{i=1}^{p} n_{i} y_{i} = n_{j},\nonumber\\
&x_{j}= 0,\nonumber\\
&\sum_{i=1}^{p} k_{i} \geq 1,\label{c2}\\
&x, y \in \Z_+^{p}, k \in \{0,1\}^{p}.\nonumber
\end{align}
\noindent where $M_{i}$ is an upper bound for the value of the $x_i$ variable in the problem, for $i=1, \ldots, p$, which was obtained when solving \eqref{p0} for $k=i$. 

Observe that by \eqref{c1} the binary $k$-variables indicate whether in the new solution the value of $x_i$ strictly increases (being $k_i=1$) or not (being $k_i=0$). Hence, by constraint \eqref{c2}, the obtained solution at least increases the value of one of the values, being then non dominated by $\underline{x}^j$. The solution of \eqref{psj} provides a lower bound for $\omega(S,n_j)$, that will be denoted by $LB_j$.
Once a solution is obtained, one can solve \eqref{EK} to check whether the new solution is efficient, and otherwise it provides an efficient solution dominating it. Then, the process is repeated to generate different efficient solutions for the problem. To assure that an already efficient solution is again generated, at the $\ell$-th iteration we incorporate to \eqref{psj}, the following sets of constraints, instead of \eqref{c1}:
\begin{equation}\label{newc}
x_{i} \leq k_{i}^{s}(\widehat{x}_{i}^{r}-1)- M_{i}(k_{i}^{r}-1), \forall i=1,\ldots, p, r=1,...\ell, 
\end{equation}
and require that $k_i^s \in \{0,1\}$, for $i=1,\ldots, p, r=1,...\ell$, where $x^1, \ldots, x^\ell \in \Z_+^p$ are the efficient solutions generated up to this iteration. Running the above problem allows us to update the lower bound if the obtained objective function is greater.

At this point, different situations are possible after solving \eqref{psj} at a certain iteration of the procedure (with the extra constraints \eqref{newc}): (1) A new solution is obtained and  $LB_j<UB_j$. In that case, a new iteration is performed; (2) \eqref{psj} with the extra constraints \eqref{newc} is infeasible. In such a case, the complete exploration have been performed, and the best solution found up that iteration is the optimal one; and (3)  $LB_j=UB_j$. In that case, since no improvement can be done on the objective function, the solution has been reached and $\omega(S,n_j)$ coincides with the computed bounds.

We show in Algorithm \ref{alg} a pseudo-code summarizing the different steps of our approach.
\begin{algorithm}[htp]
\normalsize
$\star$ \textbf{Input Data}: $S=\langle n_{1},\cdots ,n_{p}\rangle$ and $j \in \{1,\ldots,p\}$, $LB_j=0$, $\ell=1$.

$\star$  \textbf{Computation of Initial Lower Bound.}\\
Solve \eqref{p0} for $k=1, \ldots, p$ ($k\neq j$): $UB_j$ and $\overline{x}^j$.

$\star$ \textbf{Iterations:}

Solve \eqref{psj} (for $\ell>1$ with extra constraints \eqref{newc}).\\
 \If{Feasible and $ \sum_{i=1}^{p}\overline{x}^{\ell}_{i} <  UB_j $ (Situation A)}
  {
  $UB_j =\sum_{i=1}^{p}\overline{x}^{\ell}_{i}$,
     $x^{\ell+1}= \overline{x}^{\ell}$,\\
     $\ell = \ell+1$,\\
     $LB_j = \max\{ LB_j, \sum_{i=1}^{p} x_{i}^{\ell}\}$.\\
     \textbf{Solve \eqref{EK} and repeat.}
  }
  \If{\eqref{psj} is unfeasible  (Situation B) or $LB_j \geq \sum_{i=1}^{p}\overline{x}^{\ell}_{i}$ (Situation C)}
 {\textbf{Stop.}}

$\star$ \textbf{Output}: $UB_j=\omega(S,n_j)$\\
\caption{Finding the omega invariant of numerical semigroup.\label{alg}}
\end{algorithm}

\begin{proposition}
Algorithm \ref{alg} converges in a finite number of iterations.
\end{proposition}
\begin{proof}
Since the cardinality of the admissible region $\overline{D}$ is finite (the decision variables are integers and bounded) and at each iteration, the current  domain $\widetilde{D}$ is being gradually  reduced with $|\widetilde{D}|< |\overline{D}|$ until it  becomes  empty $(|\widetilde{D}|=0)$. This indicates that the method process terminates in a finite number of iterations.

Let now analyze the convergence. We suppose the existence of a feasible solution $\overline{x} \in EF(P(n_j))$; $\overline{\omega}(S,n_j)= \sum_{i=1}^{p}\overline{x}_i$ such that $\overline{\omega}(S,n_j)>  \omega^{\star}(S,n_j)$. At the terminal stage of the above algorithm, say iteration $k$, we obtain\\
        $\omega_{inf} =\omega^{\star}(S,n_j) =\sum_{i=1}^{p}x^{\star}_{i} =\sum_{i=1}^{p}x^{k}_{i} \geq  \sum_{i=1}^{p}x^{\ell}_{i}, ,  \forall \ell \in \{1, \cdots, k-1\} $.\\
        Therefore, $\overline{\omega}(S,n_j)>  \omega^{\star}(S,n_j) \geq \omega^{\ell}(S,n_j),,  \forall \ell \in \{1, \cdots,k-1\}$;
        particularly, $\overline{\omega}(S,n_j)>  \omega^{\star}(S,n_j) \geq \overline{\omega}(S,n_j)$. Therefore, $\omega^{\star}(S,n_j)=\overline{\omega}(S,n_j)$.
\end{proof}
\begin{example}\label{ex}
Let us consider the numerical semigroup $S=\langle 6,10,14,27\rangle$, and let us compute $\omega(S,27)$. First, solving \eqref{p0} for $k=1,2,3$, we get that $LB_4=9$ with $x^0=(9,0,0,0)$. Then, we solve \eqref{psj} to get a solution not dominated by $x^0$, obtaining $x^{1}= (8,6,5,0)'$ and $y^{1}= (1,1,0,5)$, with an optimal value $UB_4= 19$. Since $UB_4> LB_4$, we solve \eqref{EK} to get an efficient solution which dominates $x^1$. We obtain $\widehat{x}^{1}=(0,4,1,0)$ and $\widehat{y}^{1}=(0,0,0,1)$,  with an optimal value of $5<9$, so the lower bound is not updated. The rest of the iterations are summarized in Table \ref{table:ex}, where one can observe that after $8$ iterations one get that $LB_4=UB_4=10$, being then $\omega(S,27)=10$.
\ref{table:ex}.
\begin{table}[h]
\caption{Iterations of the algorithm to compute $\omega(S,27)$ of Example \ref{ex}.}
\label{table:ex}
\begin{center}
\normalsize\begin{tabular}{|ccccc|}\hline
It & \eqref{psj} &  \eqref{EK} & $LB_4$ & $UB_4$ \\\hline
1 & $x^{1}=( 8, 6 ,  5,  0)$&  $\hat{x}^{1}=( 0,  4 ,  1,  0)$& 9&19\\
2& $x^{2}=( 8, 3 ,  5,  0)$&  $\hat{x}^{2}=( 2 ,  0 ,  3,  0)$& 9&16\\
3& $x^{3}=( 8, 6 ,  0,  0)$&  $\hat{x}^{3}=( 0 ,  6 ,  0,  0)$& 9 &14 \\
4& $x^{4}=( 8, 5 ,  0,  0)$&  $\hat{x}^{4}=( 4 ,  3 ,  0,  0)$& 9 &13  \\
5&$x^{5}=( 8, 2 ,  2,  0)$&   $\hat{x}^{5}=( 1 ,  2 ,  2,  0)$ & 9 &12 \\
6& $x^{6}=( 8, 1 ,  2,  0)$&  $\hat{x}^{6}=( 5 ,  1 ,  1,  0)$& 9 &11 \\
7& $x^{7}=( 8, 0 ,  2,  0)$&  $\hat{x}^{7}=( 6 ,  0 ,  2,  0)$& 9  &10\\
8& $x^{8}=( 8, 2 ,  0,  0)$& $\hat{x}^{8}=( 8 ,  2 ,  0,  0)$&10&10\\
\hline
\end{tabular}
\end{center}
\end{table}
\end{example}
\section{Numerical experiments}\label{num-exp}
We have run a series of experiments in order to test the efficiency of the proposed method as well as to compare with the approach proposed by Blanco in \cite{bvictor}.

We generate a battery of random instances of different sizes. We use GAP package ``numericalsgps'' \cite{numericalsgps} 
to generate  numerical semigroups with embedding dimension $p$, with $p$ ranging in $\{ 5,7,10,12,15,17\}$ with integer generators with values running in $[100,2000]$ (calling the functions {\small \texttt{ RandomNumericalSemigroup}} and {\small\texttt{MinimalGeneratingSystemOfNumericalSemigroup}}).

We denote an instance by $(p, S(p))$, where $S(p)= \{n_1,\cdots, n_p\}$. A couple $(p, ,  S(p))$ produces a multiple objective instance  $ (p, m,n)$,
where  $p$ is the number of objectives, $m=p+1 $ the number of constraints and $n=2p$  the number of variables. CPLEX solver within MATLAB environment is used to solve these latter. For each instance $(p,S(p))$,  $10$ problems are to be solved.
\noindent The method described in section \ref{description} and the one presented in \cite{bvictor} ( referred in table \ref{compute} as Algorithm 1 and Algorithm Blanco respectively)
 were implemented  in a MATLAB environment and run on a personal computer, Intel (R) Core (TM) i5 CPU  2.5 GHz.  CPLEX 12.6  solver is also used  to solve linear  and  integer linear programming problems.

The average CPU times obtained with both approaches are reported in Table \ref{compute}. We report the minimum, maximum and median CPU times when running the $10$ instances for each $p$.

\begin{center}
\begin{table}[h]
\normalsize\centering\begin{tabular}{|l|ll|ll|}\hline
\multirow{2}{*}{$S(p)$} & \multicolumn{2}{c}{Algorithm \ref{alg} CPUTime (s)} &  \multicolumn{2}{|c|}{Algorithm \cite{bvictor} CPUTime (s)}\\
& Median & (Max,Min) &Median & (Max,Min) \\\hline
$S(5)$&  2.88 &(0.56, 120.56) & 9.27 & (1.80,16.98)\\
$S(7)$& 7.92 &(1.09, 12.19)&  17.17 & (6.40,31.95)\\
$S(10)$& 48.80& (18.98, 97.80) &  129.97 & (46.36,287.36)\\
$S(12)$& 87.23 &(52.83, 17.17)&  160.96 & (91.10,307.47)\\
$S(15)$& 257.06& (115.09, 513.84)&  357.37 & (169.45,850.22)\\
$S(17)$& 883.49& (229.15, 1721.98)& 2390.11 & (497.18,5287.36)\\
\hline
\end{tabular}
\captionsetup{justification=centering}
\centering\caption{Computational results of both algorithms.\label{compute}}
\end{table}
\end{center}
We can observe from Table \ref{compute} the differences between the proposed modified version of the method and the algorithm that is proposed in \cite{bvictor}. The improvements proposed in this paper clearly outperform the previous methodology proposed for solving the $\omega$-primality problem on the randomly generated instances. The main reason is that the new approach avoids many iterations for  each problem \eqref{psj} by using the knowledge about the initial bounds computed for the problem. Hence, the number of efficient solutions explored during the execution of the algorithm is much smaller than those generated in \cite{bvictor}. 
\section{Conclusion}\label{conclusion}

In this work, present an improved resolution method for an interesting problem in the Theory of Numbers, which is attracting the researchers in the last years,  the $\omega$-invariant of a numerical semigroup. In order to avoid the complete enumeration of a large discrete feasible set, a mathematical programming approach is proposed in \cite{bvictor}. There, the problem was formulated as a problem of optimizing a linear function over the set of efficient solutions of a multiobjective integer linear problem. However, a simple algorithm was proposed, which does not fully exploit the nature of the problem. Here, we apply a similar strategy than that proposed in \cite{bvictor} in which an easy-to-handle lower and upper bound computation of the problem allows to prune the search of the optimal solution of the problem. It allows to avoid many of the steps (related with solving integer programming problems), and then, is reflected on the CPU times needed to solve the problem.

Many extensions are possible under the topic of Mathematical Optimization tools applied to Theory of Numbers, in particular in Numerical Semigroups or in Affine Semigroups. In particular, the formulation and resolution of other algebraic indices, as the catenary degree or the tame degree \cite{blanco} will be the topic of forthcoming papers.

\section*{Acknowledgements}

The third author was partially supported by the projects MTM2016-74983-C2-1-R (MINECO, Spain), PP2016-PIP06 (Universidad
de Granada) and the research group SEJ-534 (Junta de Andaluc\'ia).

\end{document}